\documentclass[11pt]{amsart}

\usepackage[a4paper,margin=25mm,top=30mm]{geometry}

\usepackage{amsfonts, amsmath, amssymb, amsgen, amsthm, amscd}
\usepackage{newtxtext,newtxmath}
\usepackage[utf8]{inputenc} 

\usepackage{color}

\usepackage[all]{xy}

\usepackage{hyperref}
\usepackage{mathtools,slashed}

\usepackage{enumerate}

\def\a{\alpha}
\def\b{\beta}
\def\d{\delta}

\def\p{\partial}
\def\lra{\longrightarrow}

\def\ot{\otimes}
\def\odots{\ot\cdots\ot}

\def\lra{\longrightarrow}

\def\rt{\triangleright}
\def\lt{\triangleleft}

\def\lbiprod{{>\!\!\!\triangleleft\kern-.33em\cdot\, }}
\def\rbiprod{{\cdot\kern-.33em\triangleright\!\!\!<}}

\newcommand{\ps}[1]{~\hspace{-4pt}_{^{(#1)}}}

\usepackage{enumerate}

\newcommand{\G}[1]{\mathfrak{#1}}
\newcommand{\C}[1]{\mathcal{#1}}
\newcommand{\B}[1]{\mathbb{#1}}

\newcommand{\xra}[1]{\xrightarrow{\ #1\ }}

\newcommand{\CB}{{\rm CB}}
\newcommand{\CH}{{\rm CH}}

\newcommand{\Hom}{{\rm Hom}}

\newcommand{\tor}{{\rm Tor}}

\newcommand{\ie}{{\it i.e.\/}\ }

\newcommand{\tr}{\triangleright}
\newcommand{\tl}{\triangleleft}

\renewcommand{\leq}{\leqslant}
\renewcommand{\geq}{\geqslant}

\numberwithin{equation}{section}

\newtheorem{theorem}{Theorem}[section]
\newtheorem{proposition}[theorem]{Proposition}
\newtheorem{lemma}[theorem]{Lemma}
\newtheorem{corollary}[theorem]{Corollary}
\theoremstyle{definition}

\newtheorem{remark}[theorem]{Remark}

\title{On the Hochschild homology of smash biproducts}

\author{A. Kaygun}
\address{Istanbul Technical University, Istanbul, Turkey}
\email{kaygun@itu.edu.tr}

\author{S. Sütlü}
\address{Işık University, Istanbul, Turkey}
\email{serkan.sutlu@isikun.edu.tr}

\begin{document}

\begin{abstract}
  We develop a new spectral sequence in order to calculate Hochschild homology of smash biproducts
  (also called twisted tensor products) of unital associative algebras $A\# B$ provided one
  of $A$ or $B$ has Hochschild dimension less than 2.  We use this spectral sequence to
  calculate Hochschild homology of quantum tori, multiparametric quantum affine spaces,
  quantum complete intersections, quantum Weyl algebras, deformed completed Weyl algebras,
  and finally the algebra $M_q(2)$ of quantum $2\times 2$-matrices.
\end{abstract}

\maketitle

\section*{Introduction}

In this paper we investigate the Hochschild homology of a class of product algebras called the
\emph{smash biproducts}~\cite{CaenIonMiliZhu00} (or also referred as the \emph{twisted tensor
  products}~\cite{CapSchichlVanzura95}) that include Hopf-cross products and Ore extensions.
Then we effectively calculate Hochschild homologies of quantum tori, multiparametric (quantum) affine spaces~\cite{GuccGucc97,Wamb93}, quantum complete
intersections~\cite{BerErd08}, quantum Weyl algebras~\cite{Richard:QuantumWeyl}, deformed
completed Weyl algebras~\cite{DuroMeljSamsSkod07,MeljSkod07}, and  the quantum matrix algebra
$M_q(2)$~\cite{ArtinSchelterTate:QuantumDeformationsOfGLN,Kassel-book}.

Our strategy relies on splitting the Hochschild complex into a (twisted) product of two
Hochschild complexes induced by the decomposition of the underlying algebra.  Then we combine
the homologies of the individual pieces via a suitable spectral sequence.  In doing so, we do
not rely on any ad-hoc resolutions of the underlying algebras as bimodules over themselves.  The
key observation is that if one of the component algebras has Hochschild homological dimension
less than 2, then the spectral sequence degenerates on the $E^2$-page yielding the result
immediately.  We refer the reader to Section~\ref{sect:HomologyOfSmashBiproducts} for technical
details.

Here is a plan of the paper: In Section~\ref{sect:smash-biproducts} we recall basic definitions and tools needed to work with smash biproduct algebras, and then in Section~\ref{sect:HomologyOfSmashBiproducts} we construct our homological machinery.  Section~\ref{sect:calculation} contains our calculations.  In Section~\ref{subsect-quantum-tori} we calculate the Hochschild homologies of the quantum plane $k[x]\# k[y]$, the quantum cylinder $k[x]\# k[y,y^{-1}]$, and the quantum torus $k[x,x^{-1}]\# k[y,y^{-1}]$. 
Next, in Section \ref{subsect-multip-aff-sp} we calculate the Hochschild homology of the multiparametric affine space $S(X_\nu,\Lambda)$~\cite{GuccGucc97}, as well as a particular case of quantum Weyl algebras~\cite{Richard:QuantumWeyl}. The Hochschild homology of the quantum complete intersection algebra $C_{a,b}:= k\langle x,y\rangle/\langle x^a, xy-qyx,y^b\rangle $ for every $a,b\geq 2$~\cite{BerErd08} is computed in Section \ref{subsect-quant-comp-int}.  In Section \ref{subsect-quant-Weyl}, we deal with the quantum Weyl algebra $S(X_\mu,Y_\nu,\Lambda)$~ \cite{Richard:QuantumWeyl} in full generality.  Subsection \ref{sect:WeylAlgebras} deals with the $(\G{g}, D)$-deformed Weyl algebra $A^{\rm pol}_{\G{g}, D} := S(V) \rtimes U(\G{g})$, and its completion $A_{\G{g}, D} := \widehat{S(V)} \rtimes U(\G{g})$~\cite{DuroMeljSamsSkod07,MeljSkod07}. Although deformed Weyl algebras are Hopf-cross products, neither component is necessarily of low Hochschild homological dimension. Nevertheless, our method plays a critical role in identifying the (continuous) Hochschild homology of $A_{\G{g}, D}$ with its dense subalgebra $A^{\rm pol}_{\G{g}, D}$. In particular, we obtain the results
\cite[Thm. 2.1]{FeigFeldShoi05} and \cite[Thm. 4]{FeigTsyg83} for the for the completed Weyl algebra $A_{2n}$, since the Hochschild homology of $A^{\rm pol}_{2n}$ is already known.  The last subsection, Subsection \ref{subsect-quant-matrix}, is devoted to the Hochschild homology of the quantum matrix algebra $M_q(2)$. The fact that the algebra $M_q(2)$ is a short tower of Ore extensions~\cite{Kassel-book} makes it amenable for the tools we develop here. In fact, all quantum matrix algebras $M_q(n)$ are towers of Ore extensions~\cite{ArtinSchelterTate:QuantumDeformationsOfGLN,TorrLena00,Kassel-book}.  However, the growth of the lengths of these towers renders the arguments we use in the present paper ineffective. As such, the homology of the algebras $M_q(n)$ for $n>2$, and the quantum linear groups~\cite{ParshallWang:QuantumLinearGroups} demand a different approach, which we postpone to a subsequent paper.

\subsection*{Notation and Conventions}

We fix a ground field $k$ of characteristic 0.  All algebras are assumed to be unital and
associative, but not necessarily finite dimensional. We use $k[S]$ and $k\{S\}$ to denote
respectively the free commutative and free noncommutative polynomial algebras generated by a
set $S$.  We also use $\left<S\right>$ to denote either the $k$-vector space spanned, or the
two-sided ideal, generated by the set $S$ depending on the context.  If $G$ is a group,
$k[G]$ denotes the group algebra of $G$ over $k$.

 
\section{Smash biproducts, Hopf-cross products, and Ore extensions}\label{sect:smash-biproducts}

\subsection{The smash biproduct of algebras}~

Let us now recall from \cite{CaenIonMiliZhu00} the smash biproduct of
algebras, which are also called the twisted tensor
  products~\cite{CapSchichlVanzura95}. Let $A$ and $B$ be two
algebras, and let $R\colon B\ot A \lra A\ot B$ be a linear map written
as $R(b\ot a) =: {}^Ra\ot b^R$ making the diagrams
\begin{equation}\label{DistributiveLaw}
  \xymatrix{
    & B\ar[dr]^{1\otimes B}\ar[dl]_{B\otimes 1}\\
    B\otimes A \ar[rr]^{ R } & & A\otimes B\\
    & A \ar[ul]^{1\otimes A} \ar[ur]_{A\otimes 1}
  }
  \qquad
  \xymatrix{
    B\otimes B\otimes A \ar[r]^{B\otimes R }\ar[d]_{\mu_B\otimes A}
    & B\otimes A\otimes B \ar[r]^{ R \otimes B}
    & A\otimes B\otimes B \ar[d]^{A\otimes \mu_B}\\
    B\otimes A \ar[rr]^R  & &   A\otimes B\\
    B\otimes A\otimes A \ar[r]_{ R \otimes B} \ar[u]^{B\otimes\mu_A}
    & A\otimes B\otimes A \ar[r]_{A\otimes  R }
    & A\otimes A\otimes B \ar[u]_{\mu_A\otimes B}
  }
\end{equation}
commutative. We shall call any such map a \emph{distributive law}.  One can then show that $A\ot B$ is a unital associative algebra with the multiplication
\[
(a\ot b)(a'\ot b') := a({}^Ra') \ot (b^R)b',
\]
for any $a,a'\in A$ and $b,b' \in B$, and with the unit $1\ot 1 \in A \ot B$,  if and only if the diagrams in \eqref{DistributiveLaw} commute, \cite[Thm. 2.5]{CaenIonMiliZhu00}.  We will use $A\#_R B$ to denote this algebra.

Let us note that the natural inclusions
\[
i_A:A \lra A\#_R B, \quad a \mapsto a\ot 1, \qquad i_B:B \lra A\#_R B, \quad b \mapsto 1\ot b
\]
are algebra maps, and that the map
\[
\Psi:=\mu_{A\#_R B} \circ (i_A \ot i_B) : A\ot B \lra A\#_R B
\]
is an isomorphism of vector spaces. Then, the map $R$ can be recovered as
\[
R = \Psi^{-1} \circ \mu_{A\#_R B} \circ (i_B \ot i_A).
\]

Let us next review a few (more concrete) examples of this construction.

\subsection{Ore extensions}~\label{subsect:OreExtensions}

Let $A$ be an algebra, and $k[X]$ be the polynomial algebra in one indeterminate. Let
$\a:A\to A$ be an algebra homomorphism, and $\delta:A\to A$ be a derivation. Then the
invertible distributive law defined as
\[ R\colon k[X]\ot A \lra A\ot k[X], \qquad R(X \ot a) := \a(a) \ot X + \delta(a) \ot 1 \]
yields the Ore extension associated to the datum $(A,\a,\delta)$, \ie
$A\#_R k[X] \cong A[X,\a,\delta]$, \cite[Ex. 2.11]{CaenIonMiliZhu00}.


We refer the reader to \cite{FaddReshTakh89, Kassel-book} for an account of the Ore extensions in the
quantum group theory.

\subsection{Hopf-cross products}\label{HopfSmashProduct}~

We note from \cite[Ex. 4.2]{CaenIonMiliZhu00} that the smash biproduct construction covers
Hopf-cross products, i.e. algebras of the form $A \ot H$ where $H $ is a Hopf algebra and
$ A$ is a (left) $H$-module algebra.  The latter means that there is an action
$\rt \colon H \otimes A\to A$ which satisfies
\begin{equation}\label{ModuleAlgebra}
   h\rt (ab) = (h_{(1)}\rt a)(h_{(2)}\rt b) 
\end{equation}
for any $h\in H $, and any $a,b\in A$.  Then, there is an algebra structure on $A \rtimes H := A \ot H$ given by  
\[
(a\ot h)(b\ot g) = a(h\ps{1}\rt b) \ot h\ps{2}g,
\]
for any $a,b \in A$, and any $h,g \in H$, with unit $1 \ot 1 \in A \rtimes H$. Now letting
\begin{equation}
  \label{HopfDistributive}
  R\colon H \otimes A\to  A\otimes H, \qquad R(h\otimes a) = h_{(1)}\tr a\otimes h_{(2)} 
\end{equation}
for any $a\in A$ and any $h\in H$ it follows from \eqref{ModuleAlgebra} that
$A \rtimes H = A\#_R H$.

\subsection{Algebras with automorphisms}~\label{GroupAction}

Let $A$ be any unital associative algebra, and let $G$ be a group acting on $A$ via
automorphisms.  Consider the distributive law $R\colon k[G]\otimes A\to A\otimes k[G]$ given
by
\begin{equation}
  R(g\otimes a) = (g\tr a)\otimes g
\end{equation}
for every $a\in A$ and $g\in G$.  Let us denote the smash biproduct algebra $A\#_R k[G]$
simply by $A\# G$.  This is a special case of the smash product given by the distributive
law~\eqref{HopfDistributive} for the Hopf algebra $H=k[G]$.

\subsection{Algebras with derivations}~\label{AlgebrasWithDerivation}

Let $\G{g}$ be a Lie algebra, and $A$ a unital associative algebra on which $\G{g}$ acts by derivations; that is,
\[
X\tr (ab) = (X\tr a)b + a(X\tr b) 
\] 
for any $a,b\in A$, and any $X\in\G{g}$. Then, 
\[ 
R\colon U(\G{g})\otimes A\to A\otimes U(\G{g}), \qquad R(X\otimes a) = (X\tr a)\otimes 1 + a\otimes X 
\] 
determines an invertible distributive law, and hence the smashed biproduct algebra $A \#_R U(\G{g})$ which we simply denote by $A \# \G{g}$. 

Let us note that the above conditions endows $A$ with a $U(\G{g})$-module algebra structure, that is, the algebra $A \# \G{g}$ is a special case of the Hopf crossed product $A\rtimes H$ with $H=U(\G{g})$, and the distributive law~\eqref{HopfDistributive}.

\section{Homology of smash biproducts}\label{sect:HomologyOfSmashBiproducts}
\subsection{The bar complex}~

Given an associative algebra $A$, together with a right $A$-module $V$ and
a left $A$-module $W$, one can form the two sided bar complex
\begin{equation}\label{bar-complex}
\CB_*(V,A,W) := \bigoplus_{n\geq 0} V\otimes A^{\otimes n}\otimes W 
\end{equation}
together with the differentials $d\colon \CB_n(V,A,W)\lra \CB_{n-1}(V,A,W)$, $n\geq 1$, defined as
\begin{align}\label{bar-diff}
d(v\otimes  a_1\otimes\cdots\otimes a_n\otimes w) 
 &=  v\cdot a_1\otimes a_2\otimes\cdots\otimes a_n\otimes w \nonumber\\
  & + \sum_{j=1}^{n-1}(-1)^j v\otimes\cdots\otimes a_j a_{j+1}\otimes\cdots\otimes w \\
  & + (-1)^n v\otimes a_1\otimes\cdots \otimes a_{n-1}\otimes a_n\cdot w.\nonumber
\end{align}

\subsection{Hochschild homology}~

Now, let $A$ be a $k$-algebra and $V$ be an $A$-bimodule, or equivalently, a right module over the enveloping algebra $A^e:= A \ot A^{\rm op}$. Then the Hochschild homology of $A$, with coefficients in $V$, is defined to be
\[
H_\ast(A,V) := \tor^{A^e}_\ast(A,V).
\]
Equivalently, see for instance \cite[Prop. 1.1.13]{Loday-book}, the Hochschild homology of $A$, with coefficients in the $A$-bimodule $V$, is the homology of the complex
\[
\CH_\ast(A,V) := \bigoplus_{n\geq 0}\CH_n(A,V), \qquad \CH_n(A,V) := V \ot A^{\ot\,n}
\]
with respect to the differential $b\colon \CH_n(A,V)\lra \CH_{n-1}(A,V)$, defined for $n\geq 1$ as
\begin{align}\label{Hochschild-diff}
b(v\otimes  a_1\otimes\cdots\otimes a_n) 
  & = v\cdot a_1\otimes a_2\otimes\cdots\otimes a_n \nonumber\\
  & + \sum_{j=1}^{n-1}(-1)^j v\ot a_1 \otimes\cdots\otimes a_j a_{j+1}\otimes\cdots\otimes a_n  \\
  & + (-1)^n a_n\cdot v\otimes a_1\otimes\cdots \otimes a_{n-1}.\nonumber
\end{align}

\subsection{Hochschild homology with invertible distributive laws}~\label{Subsect-Hochschild-smash-E1-E2}

Let $R\colon B\otimes A\to A\otimes B$ be an invertible distributive law, and $V$ an
$A\#_R B$-bimodule, that is, the diagram
\begin{equation}
  \xymatrix{
    B\otimes A\otimes V \ar@/^1pc/[dd]\ar[r] & B\otimes V \ar[rd] &
    & V\otimes A \ar[dl] & \ar[l] V\otimes B\otimes A \ar@/^1pc/[dd] \\
    & & V\\
    A\otimes B\otimes V \ar[r]\ar@/^1pc/[uu] & A\otimes V \ar[ur] &
    & V \otimes B \ar[ul] & \ar[l] V\otimes A\otimes B \ar@/^1pc/[uu]
  }
\end{equation}
is commutative. We now introduce a bisimplicial complex computing the Hochschild homology, with coefficients, of the algebra $A\#_R B$. We shall need an iterative use of the distributive law (and its inverse), as such, we make use of the notation
\[
R(b_j\ot a_i) = {}^{R_{ji}}a_i \ot b_j^{R_{ji}}, \qquad R^{-1}(a_i\ot b_j) = {}^{L_{ji}}b_j \ot a_i^{L_{ji}}.
\]

\begin{proposition}\label{prop:diagonal}
Given a smash biproduct algebra $A\#_R B$ with an invertible distributive law $R:B\ot A \to A \ot B$, and an $A\#_R B$-bimodule $V$, the isomorphism
\begin{align*}
v \ot (a_1, b_1) & \odots (a_n,b_n) \mapsto  \\
& a_1^{L_{11} \ldots  L_{n1}} \odots a_{n-1}^{L_{(n-1)(n-1)} L_{(n-1)n}} \ot a_n^{L_{nn}}\ot v \ot {}^{L_{11}} b_1\ot {}^{L_{21} L_{22}}b_2  \odots {}^{L_{n1} \ldots L_{nn}}b_n
\end{align*}
identifies the Hochschild homology complex of the algebra $A\#_R B$, with coefficients in $V$, with the diagonal of the bisimplicial complex given by $\C{B}_{p,q}(A,V,B):=A^{\ot\,p} \ot V \ot B^{\ot\,q}$ with the structure maps
\begin{align*}
  \partial^h_i & (a_1\odots a_q\ot v\ot b_1\odots b_p)\\
   := & \begin{cases}
    a_1\odots a_q\ot v\tl b_1\odots b_p & \text{ if } i=0,\\
    a_1\odots a_q\ot v \ot b_1\odots  b_ib_{i+1}\odots b_p & \text{ if } 1\leq i\leq p-1,\\
    {}^{R_{p1} }a_1\odots {}^{R_{pq} }a_q\ot b_p^{R_{p1} \ldots  R_{pq}}\tr v \ot b_1\odots b_{p-1} & \text{ if } i=p,     
  \end{cases}\\\\
  \partial^v_j & (a_1\odots a_q\ot v\ot b_1\odots b_p)\\
  := & \begin{cases}
 a_2\odots a_q\ot v\tl {}^{R_{11} \ldots R_{p1}} a_1\ot b_1^{R_{11}}\odots b_p^{R_{p1}} & \text{ if } j=0,\\
 a_1\odots a_ja_{j+1} \ot a_q\odots v\ot b_1\odots b_p & \text{ if } 1\leq j\leq q-1, \\
 a_1\odots a_{q-1}\ot a_q\tr v\ot b_1\odots b_p & \text{ if } j=q.
  \end{cases}
\end{align*}
\end{proposition}

\begin{proof}
We observe that
\begin{align*}
\partial^h_p\partial^v_0 & (a_1\odots a_q\ot v\ot b_1\odots b_p) \\
= & \partial^h_p\left(a_2\odots a_q\ot v\tl {}^{R_{11} \ldots R_{p1}} a_1\ot b_1^{R_{11}}\odots b_p^{R_{p1}}\right) \\
= & {}^{R_{p2}}a_2\odots {}^{R_{pq}}a_q\ot b_p^{R_{p1}\ldots R_{pq}}\tr\left(v\tl {}^{R_{11} \ldots R_{p1}} a_1\right)\ot b_1^{R_{11}}\odots b_{p-1}^{R_{(p-1)1}} \\
= & {}^{R_{p2}}a_2\odots {}^{R_{pq}}a_q\ot \left(b_p^{R_{p1}\ldots R_{pq}}\tr v\right)\tl {}^{R_{11} \ldots R_{p1}}a_1\ot b_1^{R_{11}}\odots b_{p-1}^{R_{(p-1)1}} \\
= & \partial^v_0\left({}^{R_{p1} }a_1\odots {}^{R_{pq} }a_q\ot b_p^{R_{p1} \ldots  R_{pq}}\tr v \ot b_1\odots b_{p-1}\right)  \\
= & \partial^v_0\partial^h_p(a_1\odots a_q\ot v\ot b_1\odots b_p).
\end{align*}
For $1 \leq i \leq p-1$ and $1 \leq j \leq q-1$, the commutativity $\partial^h_i\partial^v_j = \partial^v_j\partial^h_i$ follows from the commutativity of the second diagram in \eqref{DistributiveLaw}.
\end{proof}

Considering the spectral sequence associated to the filtration by the rows (or the columns) of the bicomplex given in Proposition \ref{prop:diagonal}, we obtain the following result.

\begin{theorem}\label{thm:HochschildOfSmashProducts}
  Let $R\colon B\otimes A\to A\otimes B$ be an invertible distributive law, and let $V$ be an $A\#_R B$-bimodule.  Then there are two spectral sequences whose $E^1$-terms are given by
\begin{equation}\label{eqn-E1}
E^1_{p,q} = H_q(A,\CH_p(B,V)),\qquad \qquad {}'E^1_{p,q} = H_p(B,\CH_q(A,V))
\end{equation}
which converge to the Hochschild homology of the smash biproduct $A\#_R B$, with coefficients in $V$.
\end{theorem}

\begin{corollary}\label{SpecialCase}
Let $\G{g}$ be a Lie algebra action on an algebra $A$ by derivations, and let $G$ be a
 group action on an algebra $B$ by automorphisms. Let also $V$ be an $A\# \G{g}$-bimodule, and $W$ a $B\# G$-bimodule. Then there are two spectral sequences such that
  \begin{equation}
    \label{eq:LieAction}
     H_{p+q}(A\# \G{g},V) \Leftarrow {}'E^1_{p,q} = H_p(\G{g},\CH_q(A,V)^{ad}),
  \end{equation}
  and
  \begin{equation}
    \label{eq:GroupAction}
H_{p+q}(B\# G,W) \Leftarrow {}'E^1_{p,q} = H_p(G,\CH_q(B,W)^{ad}). 
  \end{equation}
\end{corollary}

\subsection{Homology of smash biproducts by amenable algebras}~

We will call a unital associative algebra $B$ as \emph{amenable} if it has Hochschild
homological dimension 0, in other words $B$ is a $B^e$-flat
module~\cite{Johnson:AmenableBanachAlgebras}.  In the discrete (algebraic) case, the typical
examples are group algebras $k[G]$ of finite groups where the characteristic of the field $k$ does not
divide the order of the group $G$.  In the measurable case, as in the case dealt originally
in~\cite{Johnson:AmenableBanachAlgebras}, the typical examples are of the form
$L^1(G)$ where $G$ is an amenable group.  In particular, all compact groups and locally
compact abelian groups are amenable,~\cite[Chap.3]{Pier:AmenableGroups}.

In the sequel, we make frequent use of the notation
\begin{equation}\label{eqn-H_0}
V_B:= \frac{V}{[B,V]} = \frac{V}{\langle bv-vb\mid v\in V, b\in B \rangle}  \cong H_0(B,V)
\end{equation}
and
\begin{equation}
V^B := \{v\in V|\ bv=vb, \text{ for all } b\in B\} \cong H^0(B,V) 
\end{equation}
for every $B$-bimodule $V$.

\begin{theorem}\label{thm:amenable}
Let $A$ and $B$ be two unital associative algebras, where $B$ is amenable, and let $R\colon B\otimes A\to A\otimes B$ be an invertible distributive law. Then,
\[ 
H_n(A\#_RB,V)\cong H_n(A,V)_B 
\] 
for any $n\geq 0$, and for any $A\#_R B$-bimodule $V$.
\end{theorem}

\begin{proof}
  By Theorem \ref{thm:HochschildOfSmashProducts} we have
\[
H_\ast(A\#_RB,V) \Leftarrow {}'E^1_{p,q} = H_p(B,\CH_q(A,V)).
\]
Since $B$ is amenable, we have
\[
{}'E^1_{p,q} \cong H_p(B,\CH_q(A,V)) \cong \begin{cases}
\CH_q(A,V)_B & \text{ if } p =0,\\
0 & \text{ otherwise}.
\end{cases}
\]
Furthermore, since $B$ is flat as a left $B^e$-module, the functor $(\ \cdot\ )\otimes_{B^e}B$ is exact. As such,
\[ 
H_\ast(A\#_RB,V) \Leftarrow {}'E^2_{p,q} = \begin{cases}
H_q(A,V)_B & \text{ if } p =0,\\
0 & \text{ otherwise}.
\end{cases}\] 
The result then follows from the spectral sequence consisting of only one column.
\end{proof}

In particular, for $V = A$ and $B = k[G]$, where $G$ is a finite group, we have the following.

\begin{corollary}
Let $G$ be a finite group acting on a unital associative algebra $A$ by automorphisms. Then
\[ 
H_\ast(A\# G,A) \cong HH_\ast(A)_{kG} \cong HH_\ast(A)^G,
\] 
where $HH_\ast(A)^G$ denotes the space of $G$-invariants (under the diagonal action).
\end{corollary}

\subsection{Homology of smash biproducts by smooth algebras}~

A unital associative $k$-algebra $B$ is called \emph{smooth} if $B$ has Hochschild
cohomological dimension 1, \ie the kernel of the multiplication map $\mu_B:B\otimes B\to B$
is $B^e$-projective; \cite[Lemma 2.3]{Schelter:SmoothAlgebras}. Such algebras are also
referred as \emph{quasi-free}; \cite{CuntzQuillen:NonsingularityI}.  Among the most basic examples of smooth algebras, there are $k[x]$ and $k[x,x^{-1}]$, see \cite[Ex. 3.4.3]{Loday-book}. Similarly, we call an
algebra \emph{$m$-smooth} if $\Omega_{B|k}$ has Hochschild homological dimension $m$. As such, an
ordinary smooth algebra is $0$-smooth.  In \cite{vdBer98}, a smooth algebra is defined to an
algebra $B$ whose projective $B^e$-resolution is finite, and contains only finitely generated
$B^e$-modules.  The quintessential example is $B=S(V)$; the polynomial algebra with
$m+1=\dim_k(V)$-indeterminates which is $m$-smooth.

\begin{lemma}\label{lem:aux}
Let $B$ be a smooth algebra, and $V$ a $B$-bimodule. Then,
\[
H_1(B,V) \cong V^B \cong H^0(B,V).
\]
\end{lemma}

\begin{proof}
  Since $B$ is smooth
\[ 
0 \to \ker(\mu_B)\to B\otimes B \to B \to 0 
\] 
is a $B^e$-projective resolution of $B$.  Then we immediately see that
\[ 
H_1(B,V)\cong \ker\left(V\otimes_{B^e}\ker(\mu_B)\to V\right). 
\] 
The claim then follows from $\ker(\mu_B)$ being generated, as a $B^e$-module, by the elements of the form $1\otimes x - x\otimes 1\in B^e$.
\end{proof}

\begin{theorem}\label{SmoothExtensions}
  Let $A$ and $B$ be two algebras with $B$ being smooth, and let $R\colon B\otimes A\to A\otimes B$ be an invertible distributive law. Then, for any $A\#_R B$-bimodule $V$, 
  \[ 
H_n(A\#_R B,V) \cong H_n(\CH_*(A,V)_B)\oplus H_{n-1}(\CH_*(A,V)^B)
  \]
  for all $n\geq 0$. 
\end{theorem}
\begin{proof}
  In view of the smoothness of $B$, Theorem~\ref{thm:HochschildOfSmashProducts} implies that the 1st page $'E^1_{p,q}$ of the spectral sequence of Theorem \ref{thm:HochschildOfSmashProducts} consists of two columns at $p=0$ and $p=1$.  Then by \eqref{eqn-H_0} 
\[ 
'E^2_{0,q} = H_q(H_0(B,\CH_\ast(A,V))) \cong H_q(\CH_\ast(A,V)_B) ,
\]
and by Lemma~\ref{lem:aux}
\[ 
'E^2_{1,q} = H_q(H_1(B,\CH_\ast(A,V)))\cong H_q(\CH_\ast(A,V)^B) 
\]
as we wanted to show.
\end{proof}

\section{Computations}\label{sect:calculation}

\subsection{Galois extensions}\label{Galois-ext}~

Let $K/k$ be a finite Galois extension, and let $G$ be the Galois group of this extension. Accordingly, we have the $K$-algebra $K\#G$, and from Theorem~\ref{thm:amenable} we obtain
\[
H_n(K\# G,K) = H_n(K)^G. 
\] 
The homology of $K$, on the other hand, is given by  
\[
HH_n(K) = \begin{cases}
K & \text{ if } n= 0, \\
0 & \text{ otherwise},
\end{cases}
\]
when regarded as a $K$-algebra. Finally $G$ being the Galois group of the extension $K/k$, we have $K^G=k$. Hence,
\[ 
H_n(K\# G,K) =
  \begin{cases}
    k & \text{ if } n=0,\\
    0 & \text{ otherwise.}
  \end{cases}
\]

\subsection{Group rings with faithful characters}\label{subsect-group-ring}~

Let $G$ be a group and let $\sigma\colon G\to k^\times$ be a character. Let us define
\begin{equation}\label{R-distr-law}
  R\colon k[x,x^{-1}]\otimes k[G]\to k[G]\otimes k[x,x^{-1}], \qquad  R(x^n\otimes g) := \sigma(g)^n g\otimes x^n 
\end{equation}
for any $n\geq 0$, and any $g\in G$. Then \eqref{R-distr-law} obeys the diagram \eqref{DistributiveLaw}, and we have the smash biproduct algebra $k[G]\# k[x,x^{-1}]$.  Now, since $k[x,x^{-1}]$ is smooth, we have by Theorem \ref{SmoothExtensions}
\begin{align*}
H_n(k[G]\# k[x,x^{-1}])
  \cong & H_n\Big(\CH_\ast(k[G],k[G]\# k[x,x^{-1}])_{k[x,x^{-1}]}\Big)\\
    & \oplus H_{n-1}\Big(\CH_\ast(k[G],k[G]\# k[x,x^{-1}])^{k[x,x^{-1}]}\Big).
\end{align*}
On the other hand,
\[ \Big[x, g_1\otimes\cdots\otimes g_n\ot (g_0\ot x^{\ell-1})\Big] =
  (\sigma(g_0g_1\cdots g_n) - 1)g_1\otimes\cdots\otimes g_n\ot (g_0\ot x^{\ell} )\]
Now, assume $\sigma$ is faithful, \ie $\sigma(g) = 1$ if and only if $g=e$.  In
view of the faithfulness of the character we get
\[ \CH_\ast(k[G],k[G]\# k[x,x^{-1}])_{k[x,x^{-1}]} = 0, \] 
whereas
\[
\CH_\ast(k[G],k[G]\# k[x,x^{-1}])^{k[x,x^{-1}]} \cong \CH_\ast^{(e)}(k[G])\otimes k[x,x^{-1}],
\]
where
\begin{align*}
  \CH_n^{(e)} (k[G])
  := & \left<g_1 \odots g_n \ot g_0 \in  \CH_n(k[G])\mid g_1\cdots g_ng_0 = e\right>.
\end{align*}
In order to compute the homology of this subcomplex we shall need the following lemma which
goes back to Eilenberg and MacLane~\cite{EilenbergMacLane:MacLaneHomology}.
(See~\cite{SiegelWitherspoon:Hochschild} for a history of Hochschild (co)homology of group
algebras.)

\begin{lemma}\label{technical2}
  Let $G$ be an arbitrary group and consider $\CH_*(k[G])$.  Now, consider the subcomplex
  $\CH^{(e)}_*(k[G])$ generated by homogeneous tensors of the form
  $g_0\otimes\cdots\otimes g_n$ such that the product $g_0\cdots g_n$ is the unit element.
  Then
  \[ H_\ast(\CH_\ast^{(e)}(k[G]) \cong  H_\ast(k[G],k). \]
\end{lemma}

\begin{proof}
  The morphism of complexes $\CH_n(k[G],k)\to \CH_n^{(e)}(k[G],k[G])$, given by
  \[ g_1\otimes\cdots\otimes g_n \mapsto  g_1\otimes\cdots\otimes g_n \ot  g_n^{-1}\cdots g_1^{-1}, \]
  is an isomorphism.
\end{proof}

Accordingly, we have
\[
H_\ast\Big(\CH_\ast(k[G],k[G]\# k[x,x^{-1}])^{k[x,x^{-1}]} \Big) \cong H_\ast(k[G],k)\otimes k[x,x^{-1}],
\]
and hence
\[
 HH_n(k[G]\# k[x,x^{-1}]) \cong H_{n-1}(G,k)\otimes k[x,x^{-1}]. 
 \] 
 Let $G$ be a finitely generated abelian group of the form $G = G^f\times G^t$ where $G^t$
 is the maximal torsion subgroup and $G^f$ is the maximal torsion-free subgroup of $G$.
 Since $H_n(G^t,k)=0$ for $n\geq 1$, we see that the higher homology is determined $G^{f}$.
 Let $a$ be the free rank of $G$.  Then
\[ HH_n(k[G]\# k[x,x^{-1}]) \cong k^{\binom{a}{n-1}}\otimes k[x,x^{-1}]. \] 

\subsection{Smash products of Lorentz polynomials}\label{subsect-quantum-tori}~

Below, we shall consider the Hochschild homology of the smash products of polynomial algebra $k[x]$ and the Laurent polynomial algebra $k[x,x^{-1}]$ in various combinations, but always with the particular distributive law given by
\[ 
R(y^j\otimes x^i) = q^{ij} x^i\otimes y^j 
\] 
for some fixed $q\in k^\times$ which is not a root of unity. 


\subsubsection{The quantum plane}~


Let us begin with $A=k[x]$ and $B=k[y]$. In this case, for any $x^{i_1} \odots x^{i_m}  \ot x^i y^j \in \CH_m(A,A\#_R B)$, we have
\begin{align*}
  \left[y^j,x^{i_1} \odots x^{i_m}  \ot x^i \right]
  = & \left(q^{(i_1+\ldots +i_m + i)j}-1\right)x^{i_1} \odots x^{i_m}  \ot x^i y^{j},
\end{align*}
that is,
\[ \CH_m(A,A\#_R B)_B \cong \CH_m(k,B/k) \oplus \CH_m(A) \]
whereas,
\[
\CH_m(A,A\#_R B)^B \cong\CH_m(k,B).
\]
Accordingly, we see that
\begin{align*}
  HH_n(A\#_R B)
  \cong & H_n(\CH_*(A,A\#_R B)_B) \oplus H_{n-1}(\CH_*(A,A\#_R B)^B) \\
 \cong & \begin{cases}
   k[x]\oplus yk[y] & \text{ if } n=0,\\
   k[x]\oplus k[y]  & \text{ if } n=1,\\
   0                & \text{ if } n\geq 2.
   \end{cases}
\end{align*}

\subsubsection{The quantum cylinder}~

We now consider $A=k[x]$ and $B=k[y,y^{-1}]$. In this case, for any $x^{i_1} \odots x^{i_m}  \ot x^i y^j \in \CH_m(A,A\#_R B)$ so that $i_1+\ldots +i_m + i \neq 0$, we may choose $j_1,j_2 \in \B{Z}$ such that $j_1+j_2 = j$ and that $j_1 \neq 0$ to get
\begin{align*}
  \left[y^{j_1},x^{i_1} \odots x^{i_m}  \ot x^iy^{j_2} \right]
  = & \left(q^{(i_1+\ldots +i_m + i)j_1}-1\right)x^{i_1} \odots x^{i_m}  \ot x^i y^{j} ,
\end{align*}
that is,
\[
\CH_m(A,A\#_R B)_B \cong \CH_m(k,B) \cong \CH_m(A,A\#_R B)^B.
\]
Accordingly, we see that
\[ 
HH_n(A\#_R B) \cong
  \begin{cases}
    k[y,y^{-1}]& \text{ if } n=0,1,\\
    0 & \text{ if } n\geq 2.
  \end{cases}
\]

\subsubsection{The quantum torus}~

We continue with $A=k[x,x^{-1}]$ and $B=k[y,y^{-1}]$. In this case, for any $x^{i_1} \odots x^{i_m}  \ot x^i y^j \in \CH_m(A,A\#_R B)$ so that $i_1+\ldots +i_m + i \neq 0$, we may choose $j_1,j_2 \in \B{Z}$ so that $j_1+j_2 = j$ and that $j_1 \neq 0$ to get
\begin{align*}
  \left[y^{j_1},x^{i_1} \odots x^{i_m}  \ot x^iy^{j_2} \right]
  = & \left(q^{(i_1+\ldots +i_m + i)j_1}-1\right)x^{i_1} \odots x^{i_m}  \ot x^i y^{j} ,
\end{align*}
that is,
\[
  \CH_m(A,A\#_R B)_B = \CH_m(A,A\#_R B)^B
  \cong \left< x^{i_1} \odots x^{i_m}  \ot x^iy^j \mid i_1 + \ldots + i_m + i = 0\right>,
\]
Just like we did in Lemma we next identify the above complexes with the complex computing the homology of the group algebra $k[\B{Z}]$ with coefficients in $k[y,y^{-1}]$ via
\begin{align*}
  \CH_n(A,A\#_R B)_B
  \cong & \CH_n(A,A\#_R B)^B \to \CH_n(k[\B{Z}],k[y,y^{-1}]), \\
& x^{\ell_1} \odots x^{\ell_n} \ot \left(x^{-\ell_1 - \ldots -\ell_n} \ot y^j\right) \mapsto \ell_1 \odots \ell_n \ot y^j,
\end{align*}
where the right $\B{Z}$-module structure on $k[y,y^{-1}]$ is given by $y^j \lt i := q^{ij}$, and the left $\B{Z}$-module structure is trivial. As such,
\begin{align*} 
& HH_n(A\#_R B) \cong H_n\Big(\CH_\ast(A,A\#_R B)_B\Big) \oplus H_{n-1}\Big(\CH_\ast(A,A\#_R B)^B\Big) \cong \\
&\hspace{3cm} H_n(k[\B{Z}],k[y,y^{-1}]) \oplus H_{n-1}(k[\B{Z}],k[y,y^{-1}]).
\end{align*}
It is then evident, in view of \cite[Ex.III.1.1]{Brown:CohomologyOfGroups}, that
\[ 
HH_n(A\#_R B) \cong
  \begin{cases}
    k & \text{ if } n=0,\\
    k \oplus k    & \text{ if } n=1,\\
    0 & \text{ if } n\geq 2.
  \end{cases}
\]

\subsection{The multiparametric affine space}\label{subsect-multip-aff-sp}~

The multiparametric affine space $S(X_\nu,\Lambda)$ defined in \cite[Subsect. 3.1]{GuccGucc97}
is the algebra generated by $X_\nu:=\{x_1, \ldots, x_\nu\}$ subject to the relations
\begin{equation}\label{QuantumPlane1}
  x_jx_i = q_{i,j}x_ix_j,
\end{equation}
where for $1 \leq i,j \leq \nu$, and the structure constants $\Lambda:=(q_{i,j})$ form a set
of nonzero elements in $k$ satisfying
\begin{equation}\label{QuantumPlane2}
q_{i,i} = 1, \qquad q_{i,j}q_{j,i} = 1,
\end{equation}
for all $i < j$. For the sake of simplicity, we are going to drop $\Lambda$ from the notation.

Now, we can write $S(X_\nu)$ as an iterated sequence of smash biproducts
\[ S(X_\nu) = \underbrace{k[x_1]\# \cdots \# k[x_n]}_{\text{$\nu$-times}} = S(X_{\nu-1})\#
  k[x_\nu] \] In view of Theorem \ref{SmoothExtensions}, we have
\begin{align*}
 HH_n(S(X_\nu))
  \cong  & H_n\Big(\CH_\ast\big(S(X_{\nu-1}),S(X_\nu)\big)_{k[x_\nu]}\Big)
         \oplus H_{n-1}\Big(\CH_\ast\big(S(X_{\nu-1}),S(X_\nu)\big)^{k[x_\nu]}\Big).
\end{align*}

\subsubsection{The free case}~

Now, let us assume for $1 \leq i,j \leq \nu$ that $\Lambda$ generates a free abelian group of rank $\nu(\nu-1)/2$ in $k^\times$. Then,
\[
\CH_\ast\big(S(X_{\nu-1}),S(X_\nu)\big)_{k[x_\nu]} \cong \CH_\ast\big(S(X_{\nu-1}),S(X_{\nu-1})\big),
\]
and
\[
\CH_\ast\big(S(X_{\nu-1}),S(X_\nu)\big)^{k[x_\nu]} \cong \CH_\ast\big(k,k[x_\nu]\big).
\]
It is then immediate to see that
\[
  HH_n(S(X_\nu)) \cong
  \begin{cases}
    k\oplus \bigoplus_{i=1}^\nu x_ik[x_i] & \text{ if } n=0, \\
    \bigoplus_{i=1}^\nu\,k[x_i] \ot x_i   & \text{ if } n=1, \\
    0                                    & \text{ if } n\geq 2,
\end{cases}
\]
which is the case $r=0$ in \cite[Thm. 4.4.1]{Richard:QuantumWeyl}.

\subsubsection{The non-free case}~

Hochschild complex admits the $\B{N}^\nu$-grading on $S(X_\nu,\Lambda)$ where we have
\[
HH_\ast(S(X_\nu)) = \bigoplus_{(m_1,\ldots,m_\nu) \in \B{N}^\nu}\,HH_\ast^{(m_1,\ldots,m_\nu)}(S(X_\nu)),
\]
The same grading applies to both $\CH_\ast\big(S(X_{\nu-1}),S(X_\nu)\big)_{k[x_\nu]}$ and
$\CH_\ast\big(S(X_{\nu-1}),S(X_\nu)\big)^{k[x_\nu]}$.  Then it suffices to consider only the
components with multi-degree $(m_1,\ldots,m_\nu) \in \B{N}^\nu$.

Now, consider the cone $C$ in the lattice $ \B{N}^\nu$ of elements
$(m_1,\ldots,m_\nu) \in \B{N}^\nu$ with the property that
\begin{equation}\label{eqn-prod}
  \prod_{i=1}^\nu \, q_{j,i}^{m_i} = 1
\end{equation}
for every $1\leq j\leq \nu$. This translates into
\begin{equation}\label{eqn-commutation}
  x_j \cdot \left(x_1^{m_1}\ldots x_\nu^{m_\nu}\right)
  = \left(x_1^{m_1}\ldots x_\nu^{m_\nu}\right)\cdot x_j
\end{equation}
for $1\leq j\leq \nu$.  Hence, it follows at once from \eqref{eqn-prod} and
\eqref{eqn-commutation} that
\[
  \dim_k HH_n^{(m_1,\ldots,m_\nu)}(S(X_\nu)) =
  \begin{cases}
    \binom{h(m_1,\ldots,m_\nu)}{n} & \text{ if } (m_1,\ldots,m_\nu)\in C\\
    0               & \text{ otherwise,}
  \end{cases}
\]
where we define
\[ h(m_1,\ldots,m_\nu) = \# \{ m_i\mid m_i > 0 \}, \] i.e. it counts the number of $m_i>0$.
Our result agrees with \cite[Thm. 3.1.1]{GuccGucc97}, and also with \cite[Thm. 6.1]{Wamb93}.

\subsection{Quantum complete intersections}\label{subsect-quant-comp-int}~

Given two integers $a,b\geq 2$, and $q \in k$ which is not a root of unity; let $C_{a,b}$ be
the quotient of the algebra $k\{x,y\}$ generated by two noncommuting indeterminates divided by
the two sided ideal given by the relations
\begin{equation}
  x^a,\quad yx-qxy, \quad y^b
\end{equation}
as given in \cite{BerErd08}. We note that the choice of
$a=b=2$ yields the quantum exterior algebra \cite{Bergh:QuantumExteriorAlgebras}. A linear
basis of this algebra is given by monomials of the form $x^iy^j$ where $0\leq i < a$ and
$0\leq j < b$.

Now, consider the truncated polynomial algebras $T_a := k[x]/\langle x^a\rangle$, and
$T_b := k[y]/\langle y^b\rangle$ determined by $a,b\geq 2$.  Then
\[
R:T_b\ot T_a \to T_a \ot T_b; \qquad R(y^j\otimes x^i) := q^{ij}x^i\otimes y^j,
\]
an invertible distributive law.  Thus we have $C_{a,b} \cong T_a \#_R T_b$.

The Hochschild homology of a truncated polynomial algebra $T_a:= k[x]/\left<x^a\right>$ has
the periodic resolution $(P_*,\p_*)$ where $P_n = T_a\otimes T_a$ and
\begin{equation*}
  \partial_n =
  \begin{cases}
    1\otimes x- x\otimes 1 & \text{ if $n$ is odd,}\\
    \displaystyle\sum_{i=0}^{a-1} x^i\otimes x^{a-1-i} & \text{ if $n$ is even.}
  \end{cases}
\end{equation*}
for $n\geq 1$ \cite[Sect. 5.9]{Kassel:CohomologyOfAssociativeAlgebras}, see
also~\cite[E.4.1.8]{Loday-book}. Tensoring (over $T_a \ot T_a$) with a $T_a$-bimodule $U$,
then $H_n(T_a,U)$ appears as the homology of the complex
\begin{equation}\label{eqn-truncated-poly-alg-comp}
\xymatrix{
\cdots \ar[r] & U \ar[r]^{\p_4} & U \ar[r]^{\p_3}  &U \ar[r]^{\p_2} & U \ar[r]^{\p_1} & U \ar[r] & 0,
}
\end{equation}
where for $n \geq 1$,
\begin{equation}\label{eqn-truncated-poly-alg-diff}
  \p_n(u) = 
  \begin{cases}
    xu-ux & \text{ if $n$ is odd,} \\
    \displaystyle\sum_{i=0}^{a-1}x^iu x^{a-1-i} & \text{ if $n$ is even.}
\end{cases}
\end{equation}
More precisely, for any $m\geq 1$ and
$x^{i_1}\odots x^{i_m}\ot x^iy^j \in \CH_m(T_a,T_a\#_R T_b)$ we have
\begin{align*}
  \p_n\left(x^{i_1}\odots x^{i_m}\ot x^iy^j\right) =
  \begin{cases}
    \left(q^{i_1+\cdots+i_m + i} - 1\right)\,x^{i_1}\odots x^{i_m}\ot x^iy^{j+1} & \text{ if $n$ is off,}\\
    \sum_{s=0}^{b-1}\, q^{(i_1+\cdots+i_m + i)s} \,x^{i_1}\odots x^{i_m}\ot x^iy^{b-1+j} & \text{ if $n$ is even.}
  \end{cases}
\end{align*}
We now see that $HH_q(T_a) \cong k^{a-1}$ for $q>0$, as in \cite[Prop. 5.4.15]{Loday-book}
Accordingly, by Theorem~\ref{thm:HochschildOfSmashProducts} we get
\[
  H_p(T_b,\CH_\ast(T_a,C_{a,b})) =\begin{cases}
    \CH_\ast(k,T_b/k)\oplus \CH_\ast(T_a), & \text{ if } p=0, \\
    \CH_\ast(k,T_b/\langle y^{b-1}\rangle) & \text{ if $p$ is odd,} \\
    \CH_\ast(k,T_b/k) & \text{ if $p>0$ is even.}
\end{cases} 
\]
Then the 1st page of the spectral sequence~\eqref{eqn-E1} appears as
\[ HH_\ast(C_{a,b}) \Leftarrow {}'E^1_{p,q} \cong
  \begin{cases}
    \CH_q(k,T_b/k)\oplus \CH_q(T_a) & \text{ if } p=0,\\
    \CH_q(k,T_b/\langle y^{b-1}\rangle) & \text{ if $p$ is odd,} \\
    \CH_q(k,T_b/k)& \text{ if $p>0$ is even.}
  \end{cases}
\]
Finally, since $H_q(k,T_b/\langle y^{b-1}\rangle) = 0 = H_q(k,T_b/k)$ for $q > 0$, the
vertical homology yields
\[ 
  {}'E^2_{p,q} \cong 
  \begin{cases}
    T_b/k \oplus k^a & \text{ if } p=q=0\\
    k^{a-1} & \text{ if $p=0$ and $q>0$}, \\
    T_b/\langle y^{b-1}\rangle & \text{ if $p$ is odd and $q=0$,}\\
    T_b/k & \text{ if $p>0$ is even and $q=0$},\\
    0 & \text{ otherwise.}
  \end{cases}
\]
In other words,
\[
  \dim_k HH_n(C_{a,b})
  = \begin{cases}
    a+b-1 & \text{ if } n =0, \\
    a+b-2 & \text{ if } n \geq 1,
  \end{cases}
\]
as in \cite[Thm. 3.1]{BerErd08}. 

\subsection{The deformed (completed) Weyl algebras}\label{sect:WeylAlgebras}~

Let $\G{g}$ be a finite dimensional Lie algebra, and $V$ be a vector space. Let also $S(V)$ denote the symmetric algebra on $V$, and
$\widehat{S(V)}$ be the completed symmetric algebra over $V$, \ie the algebra of
formal power series in $\dim(V)$-many variables.  We note that the algebra $\widehat{S(V)}$ is isomorphic to the dual algebra
$S(V)^\vee$ of $S(V)$ viewed as the coalgebra of polynomials over $\dim(V)$-many commuting
variables.

Let us now recall the $(\G{g}, D)$-deformed Weyl algebras from
\cite{DuroMeljSamsSkod07,MeljSkod07}.  Given $ D:\G{g} \to \Hom(V,\widehat{S(V)})$, extending it to a morphism $ D:\G{g} \to {\rm Der}(\widehat{S(V)})$ of Lie algebras~\cite[Subsect. 1.2]{MeljSkod07}, one arrives at the Hopf-cross product,
$A_{\G{g}, D} := \widehat{S(V)} \rtimes U(\G{g})$, called the $(\G{g}, D)$-deformed (completed) Weyl algebra.

We are going to identify the (continuous) Hochschild homology of the algebra
$A_{\G{g}, D}$ with the homology of its dense subalgebra $A^{\rm pol}_{\G{g}, D} := S(V) \rtimes U(\G{g})$.

\begin{proposition}\label{Completion}
If $ D(g)(v)\in \Bbbk\left<1\right>\oplus V$, for any $g\in \G{g}$ and
any $v\in V$, then there is an isomorphism of the form
  $HH_\ast(A^{\rm pol}_{\G{g}, D}) \cong HH^{cont}_\ast (A_{\G{g}, D})$, where the right hand side refers to the continuous Hochschild homology.
\end{proposition}

\begin{proof}
It follows at once from the hypothesis that, the distributive law between $\widehat{S(V)}$ and $U(\G{g})$ restricts to a
  degree preserving distributive law between $S(V)$ and $U(\G{g})$. As such, we have a subcomplex
\[
\CH_\ast \left( S(V)\rtimes U(\G{g})\right)_{(m)}\subseteq \CH_\ast \left( S(V)\rtimes U(\G{g})\right)
\]
of terms whose total degree in $S(V)$ is less than or equal to $m\in\B{N}$. This collection yields a projective system of complexes together with the natural epimorphisms
\begin{equation}\label{eq:projective-system}
  \CH_\ast \left( S(V)\rtimes U(\G{g})\right)_{(m+1)} \to \CH_\ast \left( S(V)\rtimes U(\G{g})\right)_{(m)}
\end{equation}
The system satisfies the Mittag-Leffler condition~\cite{Dimi04,Emma96} by definition.  Then we conclude from
\cite[Prop. 2 and Thm. 5]{Emma96} that
\begin{equation}\label{inv-lim}
  HH^{cont}_\ast(A_{\G{g}, D})
  \cong \lim_{\underset{m}{\longleftarrow}}  HH_\ast \left(S(V)\rtimes U(\G{g})\right)_{(m)}
\end{equation}
where the left hand side is the continuous Hochschild homology of $A_{\G{g}, D}$.

On the other hand, since the distributive law is degree preserving, the collection
$\CH_\ast \left( S(V)\rtimes U(\G{g})\right)_{(m)}$ forms also an injective system of complexes via the natural embeddings
\begin{equation}\label{eq:injective-system}
   \CH_\ast \left( S(V)\rtimes U(\G{g})\right)_{(m)} \to \CH_\ast \left( S(V)\rtimes U(\G{g})\right)_{(m+1)},
\end{equation}
which results in
\[
\CH_\ast \left( S(V)\rtimes U(\G{g})\right)
= \lim_{\underset{m}{\longrightarrow}} \CH_\ast \left( S(V)\rtimes U(\G{g})\right)_{(m)} .
\]
Since the homology $HH_\ast \left( S(V)\rtimes U(\G{g})\right)$ is bounded and finite dimensional at every degree, by Corollary~\ref{SpecialCase}, we see that there is an index $N \in \B{N}$ such that for every $m\geq N$ the natural injections in \eqref{eq:injective-system}, and therefore the natural
projections in \eqref{eq:projective-system}, induce quasi-isomorphisms.  In other words,
\[
  HH_\ast \Big( S(V)\rtimes U(\G{g})\Big) \cong HH^{cont}_\ast \Big( \widehat{S(V)}\rtimes U(\G{g})\Big)
\]
as we wanted to show.
\end{proof}

In particular, for the Weyl algebra $A^{\rm pol}_{2n}$, and its completion $A_{2n}$, in view of the fact that the Hochschild homology of $A^{\rm pol}_{2n}$ is known, see for
example~\cite[Sect. 5.10]{Kassel:CohomologyOfAssociativeAlgebras}, we get
\cite[Thm. 2.1]{FeigFeldShoi05} and \cite[Thm. 4]{FeigTsyg83} in
\begin{equation}
  HH^{cont}_m(A_{2n}) \cong HH_m(A^{\rm pol}_{2n}) = 
  \begin{cases}
    k & \text{ if } m= 2n\\
    0 & \text{ otherwise}
  \end{cases}
\end{equation}
for any $n\geq 1$ and any $m\geq 0$, as an immediate corollary to Proposition~\ref{Completion}.

\subsection{Quantum Weyl algebras}\label{subsect-quant-Weyl}~

In this subsection we are going to discuss the Hochschild homology of quantum Weyl algebras
$S(X_\mu,Y_\nu,\Lambda)$ generated by $X_\mu:=\{x_1,\ldots,x_\mu\}$ and
$Y_\nu:=\{y_1,\ldots,y_\nu\}$, subject to the relations
\begin{align}\label{comm-rel}
  x_ix_j - q_{i,j}x_jx_i, \quad
  y_iy_j - q_{i,j}y_jy_i, \quad
  y_jx_i - q_{i,j}x_iy_j, \quad
  x_iy_i - y_ix_i + 1
\end{align}
for all $i\neq j$ where $1\leq i\leq \mu$, $1\leq j\leq \nu$, and $\Lambda:=(q_{i,j})$ of
\eqref{QuantumPlane2}.  Since there is a an isomorphism of algebras of the form
$S(X_\mu,Y_\nu,\Lambda) \to S(X_\nu,Y_\mu,\Lambda^t)$ given by
\begin{align*} & x_i\mapsto - y_i, \qquad y_j\mapsto x_j,
\end{align*}
we may assume, without loss of generality, that $\nu\geq\mu$.  We also assume that $\Lambda$
generates a free abelian group of rank $\nu(\nu-1)/2$ in $k^\times$ as
in~\cite{Richard:QuantumWeyl}.  For simplicity, from now on we are going to drop the parameter
matrix $\Lambda$ from the notation.

\subsubsection{The case $\mu=\nu$}~

Let $\nu=\mu$, and let $A_2(x_i,y_i)$ be the copy of the ordinary Weyl algebra in
$S(X_\mu,Y_\mu)$ generated by $x_i$ and $y_i$ for $1\leq i,j\leq \mu$.  We recall also the $A_2(x_1,y_1)^e$-resolution
\[ 0 \to A_2(x_1,y_1)^e \xra{\p_2} A_2(x_1,y_1)^e\oplus A_2(x_1,y_1)^e\xra{\p_1} A_2(x_1,y_1)^e\to 0, \]
where
\[ \partial_2(a\otimes b) = (ax_1\otimes b-a\otimes x_1 b)e_1 - (ay_1\otimes b - a\otimes y_1 b)e_2, \]
and
\[ \partial_1((a_1\otimes b_1)e_1 + (a_2\otimes b_2)e_2) 
  = a_1y_1\otimes b_1 - a_1\otimes y_1 b_1 + a_2x_1\otimes b_2 - a_2\otimes x_1 b_2,
\]
 for $A_2(x_1,y_1)$, \cite[Sect. 5.10]{Kassel:CohomologyOfAssociativeAlgebras}. Since $S(X_\mu,Y_\mu) \cong S(X_{\mu-1},Y_{\mu-1}) \# A_2(x_1,y_1)$, in view of
Theorem~\ref{thm:HochschildOfSmashProducts}, we have
\[ HH_*(S(X_\mu,Y_\mu)) \Leftarrow {}'E^1_{p,q} =
  H_p\Big(A_2(x_1,y_1),\CH_q\big(S(X_{\mu-1},Y_{\mu-1}),S(X_\mu,Y_\mu)\big)\Big).
\]
Now we note that
\begin{align*}
 &  \left[y_1 , x_2^{a_2}y_2^{b_2}\odots x_\mu^{a_\mu}y_\mu^{b_\mu} \ot
  x_1^ay_1^b \right]  \\
 & =   (q_{1,2}^{B_2-A_2}\cdots q_{1,\mu}^{B_\mu-A_\mu}-1)
        x_2^{a_2}y_2^{b_2}\odots x_\mu^{a_\mu}y_\mu^{b_\mu} \ot x_1^a y_1^{b+1}\\
    & \quad + aq_{1,2}^{B_2-A_2}\cdots q_{1,\mu}^{B_\mu-A_\mu} x_2^{a_2}y_2^{b_2}\odots x_\mu^{a_\mu}y_\mu^{b_\mu} \ot x_1^{a-1}y_1^b,
\end{align*}
and that
\begin{align*}
   & \left[x_1, x_2^{a_2}y_2^{b_2}\odots x_\mu^{a_\mu}y_\mu^{b_\mu}
  \ot x_1^ay_1^b\right]\\
  & = (q_{1,2}^{A_2-B_2}\cdots q_{1,\mu}^{A_\mu-B_\mu}-1)
        x_2^{a_2}y_2^{b_2}\odots x_\mu^{a_\mu}y_\mu^{b_\mu} \ot x_1^{a+1} y_1^b\\
    &\quad + bq_{1,2}^{A_2-B_2}\cdots q_{1,\mu}^{A_\mu-B_\mu}x_2^{a_2}y_2^{b_2}\odots x_\mu^{a_\mu}y_\mu^{b_\mu} \ot x_1^ay_1^{b-1},
\end{align*}
where $A_i$ and $B_j$, $2\leq i,j \leq \mu$, denote the total degree of $x_i$ and $y_j$ respectively. Hence, upon tensoring the above resolution with the complex $\CH_*(S(X_{\mu-1},Y_{\mu-1}), S(X_\mu,Y_\mu))$ we see that
\[ {}'E^1_{p,q} \cong
  \begin{cases}
    \CH_q(S(X_{\mu-1},Y_{\mu-1}), S(X_\mu,Y_\mu))^{A_2(x_1,y_1)} & \text{ if } p=2,\\
    0 & \text{ otherwise.}
  \end{cases}
\]
The invariant subcomplex consists of the tensors whose total degree of $x_i$ terms is
equal to the total degree of $y_i$ terms for each $2 \leq i \leq \mu$, with no $x_1$ or $y_1$ terms. Moreover, this subcomplex is a direct summand (as differential graded modules) of the
Hochschild complex $\CH_q(S(X_\mu,Y_\mu))$, as such, we see that
\begin{align*}
  {}'E^2_{2,q} \cong H_q\big(\CH_\ast(S(X_{\mu-1},Y_{\mu-1}), S(X_\mu,Y_\mu))^{A_2(x_1,y_1)}\big) = 
     HH_q\left(S(X_{\mu-1},Y_{\mu-1})\right)^{A_2(x_1,y_1)} .
\end{align*}
In other words,
\[ 
HH_n(S(X_\mu,Y_\mu)) = HH_{n-2}\left(S(X_{\mu-1},Y_{\mu-1})\right)^{A_2(x_1,y_1)}. 
\]
Reducing recursively then, we obtain
\begin{equation}
HH_n(S(X_\mu,Y_\mu)) \cong \begin{cases}
k & \text{ if } n = 2\mu, \\
0 & \text{ otherwise}.
\end{cases}
\end{equation}

\subsubsection{The case $\nu>\mu$}~

Considering
\[  S(X_\mu,Y_\nu) \cong S(X_\mu,Y_{\nu-1}) \#_R k[y_\nu], \]
Theorem \ref{SmoothExtensions} yields
\begin{align*}
  HH_n\big(S(X_\mu, & Y_\nu)\big)\\
  \cong & H_n\Bigg(\CH_\ast\bigg(S(X_\mu, Y_{\nu-1}),  S(X_\mu,Y_\nu)\bigg)_{k[y_\nu]}\Bigg) \oplus H_{n-1}\Bigg(\CH_\ast\bigg(S(X_\mu, Y_{\nu-1}),  S(X_\mu,Y_\nu)\bigg)^{k[y_\nu]}\Bigg).
\end{align*}
It is evident from the commutation relations \eqref{comm-rel} that
\begin{align*}
 \CH_\ast \bigg(S(X_\mu, & Y_{\nu-1}),  S(X_\mu,Y_\nu)\bigg)_{k[y_\nu]}\\
  \cong & \CH_\ast\bigg(S(X_\mu, Y_{\nu-1}),  S(X_\mu,Y_{\nu-1})\bigg) 
          \oplus \CH_\ast\bigg(k, k[y_\nu]/k\bigg).
\end{align*}
and that the $k[y_\nu]$-invariant subcomplex $\CH_\ast\big(S(X_\mu, Y_{\nu-1}),  S(X_\mu,Y_\nu)\big)^{k[y_\nu]}$
may be written as a direct sum of $\CH_\ast\big(k,   k[y_\nu]\big)$
and a complex with the same total degree of $x_i$'s and $y_i$'s, for each $1 \leq i \leq \mu$. However, since these complexes are accompanied with an extra $\ot y_{\nu}$, the reduction performed in the previous subsection reveals that the homology (on the next page) of the latter is trivial. As a result,
\begin{align*}
  HH_n(S(X_\mu,Y_\nu))
  \cong & HH_n(S(X_\mu,Y_{\nu-1})) \oplus H_n\Big(\CH_\ast\big(k, k[y_\nu]/k\big)\Big)\\
  & \oplus H_{n-1}\Big(\CH_\ast\big(k,   k[y_\nu]\ot y_\nu\big)\Big).
\end{align*}
Hence, inductively, we see that
\[
  HH_n(S(X_\mu,Y_\nu))
  \cong \begin{cases}
    \bigoplus_{j=\mu+1}^\nu\,y_jk[y_j]
    & \text{ if }  n=0, \\
    \bigoplus_{j=\mu+1}^\nu\,k[y_j] \ot y_j
    & \text{ if }  n=1, \\
    k & \text{ if }  n= 2\mu, \\
    0 & \text{ otherwise},
  \end{cases}
\]
as in \cite[Thm. 4.4.1]{Richard:QuantumWeyl}.

\subsection{The algebra of quantum matrices}\label{subsect-quant-matrix}~

Given $q \neq 1$, the algebra $M_q(2)$ of quantum matrices is defined in \cite[Def. IV.3.2]{Kassel-book} as the quotient algebra generated by $\{a,b,c,d\}$ subject to the relations
\begin{align*}
ba & = qab, &       db & = q bd, \\
ca & = qac, &       dc & = qcd, \\
bc & = cb,  &  ad - da & = (q^{-1} - q)bc.
\end{align*}
We note also that the center of $M_q(2)$ is generated by the quantum determinant $D_q := ad - q^{-1}bc \in M_q(2)$, see \cite[Thm. 1.6]{NoumYamaMima93}).

More importantly, $M_q(2)$ is a tower of Ore extensions $A_1\subseteq A_2\subseteq A_3\subseteq M_q(2)$ given by
\[ A_1 = k[a], \quad A_2 \cong A_1[b,\a_1,0], \qquad A_3 \cong A_2[c,\a_2,0], \qquad M_q(2) \cong A_3[d,\a_3,\d], 
\]
see, for instance, \cite[Sect. IV.4]{Kassel-book}. The structure morphisms $\alpha_i\colon A_i\to A_i$ are given as
\begin{align*}
\a_1(a)       & := qa, \\
\a_2(a)       & := qa, & \a_2(b) &:= b, \\
\a_3(a)       & := a,  & \a_3(b) &:= qb, & \a_3(c) &:= qc.
\end{align*}
for $i=1,2,3$, and $\delta\colon A_3\to A_3$ is given as
\begin{align*}
\d(b^jc^k)    & := 0 & 
\d(a^ib^jc^k) & := (q - q^{-1})\frac{1-q^{2i}}{1-q^2}a^{i-1}b^{j+1}c^{k+1}.
\end{align*}
Let us also recall from \cite[Lemma 1]{GeLiuSun92} the relations
\begin{equation}\label{eqn-ad-da-commutation}
d^na = ad^n - q\left(1-q^{-2n}\right)bcd^{n-1}, \qquad da^n = a^nd  + q\left(q^{-2n}-1\right)a^{n-1}bc.
\end{equation}
It follows from Theorem~\ref{thm:HochschildOfSmashProducts} that
\[
HH_\ast(M_q(2)) \Leftarrow E^1_{i,j} := H_j(A_3,\CH_i(k[d],M_q(2))),
\]
where, setting $U_i:= \CH_i(k[d],M_q(2))$, the same results yields also
\[
  H_\ast(A_3, U_i) \Leftarrow {}'E^1_{m,n} := H_m(k[c],\CH_n(A_2,U_i))
  = \begin{cases}
    \CH_n(A_2,U_i)_{k[c]} & \text{ if } \,\, m = 0, \\
    \CH_n(A_2,U_i)^{k[c]} \ot c & \text{ if } \,\, m = 1, \\
    0 & \text{ if } \,\, m \geq 2.
    \end{cases}
\]
On the next step we note that
\[
\Big[k[c],M_q(2)\Big] = \Big\langle a^rb^sc^td^\ell \mid r\neq \ell,\ t>0 \Big\rangle,
\]
and hence,
\[
\frac{M_q(2)}{\Big[k[c],M_q(2)\Big]} := \Big\langle a^rb^sd^\ell, a^\ell b^s c^td^\ell \mid r,s,\ell\geq 0, \quad t>0 \Big\rangle.
\]
Following the ideas of Proposition \ref{prop:diagonal} we may identify the ($A_2$, or vertical) homology of the complex 
\begin{align*}
  \CH_n & (A_2,U_i)_{k[c]}\\
  \cong & \Big\langle a^{r_1}b^{s_1}\odots a^{r_n}b^{s_n} \ot a^{r_{n+1}}b^{s_{n+1}} d^{\ell_1} \ot d^{\ell_2} \odots d^{\ell_{i+1}} \mid r_u,s_v,\ell_w\geq 0 \Big\rangle \\
  & \oplus \Big\langle a^{r_1}b^{s_1}\odots a^{r_n}b^{s_n} \ot a^{r_{n+1}}b^{s_{n+1}}c^td^{\ell_1} \ot d^{\ell_2} \odots d^{\ell_{i+1}} \mid \\
  & \qquad r_u,s_v,\ell_w\geq 0, t > 0, \sum_u r_u = \sum_w \ell_w \Big\rangle
\end{align*}
with the total homology of the bicomplex
\begin{align*}
\CH_{\a,\b} & (A_2,U_i)_{k[c]} \\
  \cong & \Big\langle a^{r_1}\odots a^{r_\b}\ot a^{r_{\b+1}} b^s d^{\ell_1} \ot d^{\ell_2} \odots d^{\ell_{i+1}} \ot b^{s_1}\odots b^{s_\a} \mid r_u,s,\ell_w\geq 0 \Big\rangle \\
  & \oplus \Big\langle a^{r_1}\odots a^{r_\b}\ot a^{r_{\b+1}} b^sc^td^{\ell_1} \ot d^{\ell_2} \odots d^{\ell_{i+1}} \ot b^{s_1}\odots b^{s_\a} \mid \\
  & \qquad r_u,s,t,\ell_w\geq 0, t>0, \sum_u r_u = \sum_w \ell_w\Big\rangle.
\end{align*}
The homology of this bicomplex, in turn, is approximated by the spectral sequence
\begin{align}\label{eqn-homology-E1-coinv}
  H_\ast & \left(\CH_\ast(A_2,U_i)_{k[c]} \right) \nonumber
  = H_\ast\left(\CH_{\ast,\ast}(A_2,U_i)_{k[c]} \right)\\
  \Leftarrow {}'E^1_{\a,\b} & 
  \cong \begin{cases}
    \Big\langle a^{r_1}\odots a^{r_\b}\ot a^rb^sc^td^\ell \ot d^{\ell_1} \odots d^{\ell_i} \mid \\
    \qquad r,r_1,\ldots r_\b,s,t,\ell,\ell_1,\ldots,\ell_i\geq 0,\quad r_1 + \ldots + r_\b + r = \ell_1 + \ldots + \ell_i +\ell\Big\rangle\\
    \Big\langle a^{r_1}\odots a^{r_\b}\ot a^rd^\ell \ot d^{\ell_1} \odots d^{\ell_i}  \mid  r,r_1,\ldots r_\b,s_1,\ldots,s_\a,\ell,\ell_1,\ldots,\ell_i\geq 0\big\}\Big\rangle & \text{ if } \a= 0, \\
    \Big\langle a^{r_1}\odots a^{r_\b}\ot a^rb^sc^td^\ell \ot d^{\ell_1} \odots d^{\ell_i} \mid \\
    \qquad r,r_1,\ldots r_\b,s,t,\ell,\ell_1,\ldots,\ell_i\geq 0,\quad r_1 + \ldots + r_\b + r = \ell_1 + \ldots + \ell_i +\ell \Big\rangle \ot b
    & \text{ if } \a= 1, \\
    0 & \text{ if } \a\geq 2.
\end{cases}
\end{align}
Noticing, in view of \ref{eqn-ad-da-commutation}, that
\[
\Big[a, a^{\ell-1}b^sc^td^\ell\Big] = (1-q^{s+t})a^\ell b^sc^td^\ell - q(1-q^{-2\ell}) a^{\ell-1}b^{s+1}c^{t+1}d^{\ell-1},
\]
in the vertical homology we arrive at
\begin{align}\label{eqn-homology-E2-coinv}
  H_\ast & \left(\CH_\ast(A_2,U_i)_{k[c]} \right) \nonumber\\
  \Leftarrow & {}'E^2_{\a,\b}
  \cong \begin{cases}
    \CH_i(k,k[b,c]) \oplus \CH_i(k[d],k[a,d]) & \text{ if } \a= 0, \b=0, \\
    a \ot \CH_i(k,k[a])   & \text{ if } \a= 0,\ \b=1\\
    \CH_i(k,k[b,c])\ot b  & \text{ if } \a= 1,\ \b=0, \\
    0 & \text{ otherwise. }
\end{cases}
\end{align}
Now, on the other hand,
\[
M_q(2)^{k[c]} \cong \Big\langle \big\{  a^rb^s c^td^r \mid r,s,t\geq 0\big\}\Big\rangle,
\]
as such,
\begin{align*}
  \CH_n(A_2,U_i)^{k[c]}
  \cong & \Big\langle  a^{r_1}b^{s_1}\odots a^{r_n}b^{s_n} \ot a^rb^sc^td^\ell \ot d^{\ell_1} \odots d^{\ell_i} \mid \\
        &  \qquad r,r_1,\ldots r_n,s,s_1,\ldots,s_n,t,\ell,\ell_1,\ldots,\ell_i\geq 0, \quad r_1 + \ldots + r_n + r = \ell_1 + \ldots + \ell_i +\ell\Big\rangle.
\end{align*}
Similarly above, we may identify the ($A_2$, or vertical) homology of this complex with the total homology of the bicomplex
\begin{align*}
  \CH_{\mu,\nu}(A_2,U_i)^{k[c]}
  \cong & \Big\langle a^{r_1}\odots a^{r_\nu}\ot a^rb^sc^td^\ell \ot d^{\ell_1} \odots d^{\ell_i} \ot b^{s_1}\odots b^{s_\mu} \mid \\
        & \qquad r,r_1,\ldots r_\nu,s,s_1,\ldots,s_\mu,t,\ell,\ell_1,\ldots,\ell_i\geq 0, \quad r_1 + \ldots + r_\nu + r = \ell_1 + \ldots + \ell_i +\ell\Big\rangle.
\end{align*}
Now, following the same line of thought in \eqref{eqn-homology-E1-coinv} and \eqref{eqn-homology-E2-coinv}, we obtain
\begin{align*}
  H_\ast \left(\CH_\ast(A_2,U_i)^{k[c]} \ot c\right)  \Leftarrow & {}'E^2_{\a,\b} \\
  \cong & 
    \begin{cases}
      \CH_i(k,k[b,c])\ot c & \text{ if }  \a= 0, \ \b=0, \\
      \CH_i(k,k[b,c]) \ot b\ot c & \text{ if }  \a= 1,\ \b=0, \\
      0 & \text{ otherwise. } 
    \end{cases}
\end{align*}
We thus conclude,
\begin{align*}
  H_\ast(A_3, U_i)
  \Leftarrow & {}'E^2_{m,n}
              \cong \begin{cases}
                 \CH_i(k,k[b,c]) \oplus \CH_i(k[d],k[a,d])
                 & \text{ if } m=0\text{ and } n= 0, \\
                 a \ot \CH_i(k,k[a]) \oplus \CH_i(k,k[b,c]) \ot b
                 & \text{ if } m=0\text{ and } n= 1, \\
                 \CH_i(k,k[b,c]) \ot c
                 & \text{ if } m=1\text{ and } n= 0, \\
                 \CH_i(k,k[b,c]) \ot b \ot c
                 & \text{ if } m=1\text{ and } n= 1, \\
                 0 & \text{ otherwise.}
               \end{cases}
\end{align*}
As such,
\begin{align*}
  HH_\ast(M_q(2))
  \Leftarrow & E^1_{i,j} \\
  := & H_j(A_3,\CH_i(k[d],M_q(2))) \\
  \cong & \begin{cases}
    \CH_i(k,k[b,c]) \oplus \CH_i(k[d],k[a,d])
    & \text{ if } j = 0, \\
    a \ot \CH_i(k,k[a]) \oplus \CH_i(k,k[b,c]) \ot b \oplus \CH_i(k,k[b,c]) \ot c
    & \text{ if } j = 1, \\
    \CH_i(k,k[b,c]) \ot b \ot c
    & \text{ if } j = 2, \\
    0 & \text{ if } j \geq 3,
  \end{cases}
\end{align*}
and
\begin{align*}
  HH_\ast(M_q(2)) \Leftarrow & E^2_{i,j} \\
  \cong & \begin{cases}
    k[b,c] \oplus k[a,d]
    & \text{ if } j = 0 \text{ and } i =0, \\
    k[d] \ot d
    & \text{ if } j = 0 \text{ and } i =1, \\
    a \ot k[a] \oplus k[b,c] \ot b \oplus k[b,c] \ot c
    & \text{ if } j = 1 \text{ and } i =0, \\
    k[b,c]\ot b \ot c
    & \text{ if } j = 2 \text{ and } i =0, \\
    0 & \text{ if } j \geq 3 \text{ or } i\geq 1.
  \end{cases}
\end{align*}
In other words, we obtain
\[
  HH_n(M_q(2)) \cong
  \begin{cases}
    k[b,c] \oplus k[a,d]
    & \text{ if } n =0, \\
    k[a] \oplus k[b,c] \oplus k[b,c] \oplus k[d]
    & \text{ if } n =1, \\ 
    k[b,c]
    & \text{ if } n =2, \\
    0 & \text{ if } n \geq 3.
  \end{cases}
\]

\begin{remark}
  We note that the result above does not follow \cite[Coroll. 2.5]{GuccGucc97} since the
  extension $M_q(2)=A_3[d,\a_3,\d]$ does not satisfy the hypothesis therein.
\end{remark}

\bibliographystyle{plain}
\bibliography{references}{}

\end{document}